\documentclass[pdflatex,sn-mathphys-num]{sn-jnl}

\usepackage{graphicx}%
\usepackage{multirow}%
\usepackage{amsmath,amssymb,amsfonts}%
\usepackage{amsthm}%
\usepackage{mathrsfs}%
\usepackage[title]{appendix}%
\usepackage{xcolor}%
\usepackage{textcomp}%
\usepackage{manyfoot}%
\usepackage{booktabs}%
\usepackage{algorithm}%
\usepackage{algorithmicx}%
\usepackage{algpseudocode}%
\usepackage{listings}%
\usepackage{amsmath,amssymb,amsfonts}%
\usepackage{amsthm}%
\usepackage{mathrsfs}%
\usepackage[title]{appendix}%
\usepackage{xcolor}%
 \usepackage{manyfoot}%
\usepackage[all,2cell]{xy}
 \usepackage{tikz-cd}
\usepackage{amsfonts}
 \input diagxy
 \usepackage{csquotes}
 \usepackage{cancel}
\usepackage{appendix}
\usepackage{lipsum}
 
\usepackage{indentfirst}
\usepackage{mathrsfs}
\usepackage{extarrows}
\usepackage{mathtools}
\usepackage{hyperref}

\theoremstyle{thmstyleone}%
\newtheorem{theorem}{Theorem}[section]
\newtheorem*{theorem*}{Theorem}
\newtheorem{definition}[theorem]{Definition}
\newtheorem{remark}[theorem]{Remark} 
\newtheorem{remarks}[theorem]{Remarks}
\newtheorem{lemma}[theorem]{Lemma}

\newtheorem{proposition}[theorem]{Proposition}
\newtheorem{corollary}[theorem]{Corollary}

\newtheorem{Fubini's theorem}[theorem]{Fubini's theorem}
\newcommand{\K}{\mathsf{K}} 
\newcommand{\Loc}{\mathsf{Loc}}
\newcommand{\KHLoc}{\mathsf{KHLoc}}
\newcommand{\HLoc}{\mathsf{HLoc}}

 \newcommand{\KRLoc}{\mathsf{KRLoc}}
 
\newcommand{\kHLoc}{\mathsf{kHLoc}}

\newcommand{\CRLoc}{\mathsf{CRLoc}}
\newcommand{\Set}{\mathsf{Set}}
\newcommand{\Top}{\mathsf{Top}}
\newcommand{\HTop}{\mathsf{HTop}}

\newcommand{\Hom}{\mbox{Hom}}

\newcommand{\xra}{\xrightarrow}
\newcommand{\sra}{\rightarrow}
\newcommand{\ra}{\longrightarrow}

\newcommand{\Ra}{\Rightarrow}

\newcommand{\st}{\stackrel}
\newcommand{\tm}{\times}
\newcommand{\colim}{\mathsf{colim}}
\newcommand{\lm}{\mathsf{lim}}

\newcommand{\ma}{\mathcal{A}}
\newcommand{\mk}{\mathcal{K}}
\newcommand{\mw}{\mathcal{W}}
\newcommand{\mi}{\mathcal{I}}
\newcommand{\mj}{\mathcal{J}}

\newcommand{\mb}{\mathcal{B}}
\newcommand{\mc}{\mathcal{C}}

\newcommand{\mr}{\mathcal{R}}
\newcommand{\mq}{\mathcal{Q}}
\newcommand{\mpp}{\mathcal{P}}

\newcommand{\msl}{\mathcal{S}\ell}

 \newcommand{\longrightsquigarrow}{\relbar\joinrel\rightsquigarrow}

 \newcommand{\oppair}[4]{\xymatrix{#1 \ar@<.5ex>[r]^-{#3} &{#2}\ar@<.5ex>[l]^-{#4}}}
\newcommand{\oppairi}[4]{\xymatrix@1{#1 \ar@<.5ex>[r]^{#3} &{#2}\ar@<.5ex>[l]^{#4}}}

\begin{document}

\title[Article Title]{A convenient category of locales}


\author*[1]{\fnm{Moncef } \sur{Ghazel}}\email{moncef.ghazel@fst.utm.tn}

\author[1]{\fnm{Inès} \sur{Saihi}}\email{ines.saihi@fst.utm.tn}
\equalcont{These authors contributed equally to this work.}

\author[2]{\fnm{Walid} 
\sur{ Taamallah}}\email{walid.taamallah@ipeiem.utm.tn}
\equalcont{These authors contributed equally to this work.}

\affil*[1]{\orgdiv{Faculté des Sciences de Tunis}, \orgname{University of Tunis El Manar}, \orgaddress{\street{Rue de Tolède}, \city{Tunis}, \postcode{2092}, \state{Tunis}, \country{Tunisia}}}

\affil[2]{\orgdiv{Institut Préparatoire aux Études d’Ingénieurs El Manar}, \orgname{University of Tunis El Manar}, \orgaddress{\street{BP 244}, \city{Tunis}, \postcode{2092}, \state{Tunis}, \country{Tunisia}}}


\abstract{The concept of Kan extendable subcategories was originally introduced to define the category of compactly generated fibrewise topological spaces over a $T_1$
base space and to establish its cartesian closure. In this paper, we show that the same framework can likewise be applied to define the category of compactly generated strongly Hausdorff locales and to prove that it, too, is cartesian closed.}




\maketitle
\section{Introduction}
In many branches of mathematics, work is conducted within a certain category. It  often happens that the category under study lacks important categorical properties, making it difficult to construct objects and morphisms. The search for good alternatives to such categories has stimulated extensive research over the last fifty years \cite{BH,  HSS, RC, RE, SN, SP, GM}. A replacement of the original category is regarded as convenient if it satisfies the essential properties of a category—namely, being bicomplete and cartesian closed—and contains models for the interesting examples of the original category.

In the literature, there are essentially three ways to construct a good replacement for a category lacking the desired properties. The first consists of using the good objects of the original category to generate a new well-behaved subcategory \cite{SN, SP,RC,BH,GM}. The second method is to find a fully faithful embedding of the original category into a good one. An example is the Yoneda embedding which injects any locally small category $\mc$ into  the category of presheaves over $\mc$  which is bicomplete and cartesian closed   \cite[Corollary 2.2.8]{RE}. The third method is to construct a completely new convenient category that can be related to the original category in an appropriate way \cite{HSS}.

Topology concerns the category of topological spaces and continuous maps. Although bicomplete, the category of topological spaces, denoted by $\Top$, has certain shortcomings: besides not being cartesian closed, in this category, products of quotient maps need not be quotient maps. The category of compactly generated topological spaces and its variants is perhaps one of the earliest and most prominent examples of convenient categories \cite{RC, BR, SN, SP}. It restores cartesian closure, and binary products in it respect quotient maps.

The first author extended the notion of compactly generated spaces to the fibrewise setting \cite{GM}. He then proved that these spaces form a cartesian closed category. We now give a brief outline of his method.

Let $\mw$ be any subcategory of a category $\mc$. Assume that the inclusion functor $J: \mw \to \mc$ has an idempotent density comonad $T$, and let $\mw_l[\mc]$ be the category of $T$-coalgebras. Then the forgetful functor $U:\mw_l[\mc] \to \mc$ is a fully faithful embedding \cite[Proposition 2.2]{GM}. Therefore, the category $\mw_l[\mc]$ can be viewed as a full subcategory of $\mc$. Assume that the objects of $\mw$ are exponentiable in $\mc$. Then, under  mild conditions, it can be shown that $\mw_l[\mc]$ is cartesian closed \cite[Theorem 9.6]{GM}. By taking $\mc$ to be the category $\Top$ and $\mw$ to be the category of compact Hausdorff spaces, one obtains the category of compactly generated spaces as $\mw_l[\mc]$. The same considerations apply to the fibrewise situation.
Similarly, by taking $\mc$ to be the category $\HTop$ of Hausdorff spaces and $\mw$ to be the category of compact Hausdorff spaces, one obtains the category of compactly generated Hausdorff spaces defined by Steenrod in \cite{SN} as $\mw_l[\mc]$.

Point-free topology is an approach to topology based on the study of the complete lattice formed by the open sets of a given topological space \cite{PP, PPS, PPT}.
Complete lattices
satisfying the infinite distributive law
and their appropriate morphisms form the category of locales, denoted by $\Loc$. It is closely related to the category of topological spaces and allows for topological constructions using purely algebraic methods.

Despite being bicomplete, the category $\Loc$ inherits certain deficiencies from $\Top$. Notably, $\Loc$ fails to be cartesian closed, for by a theorem of Hyland, a locale is exponentiable iff it is continuous \cite[Theorem 1]{HM}. This fact highlights the need for a point-free counterpart of the category of compactly generated Hausdorff spaces described above. The significance of introducing such a category was already emphasized by Johnstone \cite{J}, and the first steps in this direction were taken by Escardó \cite{EM}.

Our main objective in this work is to show that the method   described above can be used to define the category of compactly generated strongly Hausdorff locales, and to establish its cartesian closure. This category will be denoted by $\kHLoc$.
It should be noted that our proof of the cartesian closedness of 
$\kHLoc$  (Theorem \ref{x46}) is non-constructive, as it leans on a classical fact about exponentials of locales. Whether a constructive proof is possible remains open.

The paper is organized as follows: In Section \ref{s1}, we briefly review the notion of Kan-extendable subcategories and summarize their main properties. Section \ref{s2}  contains a brief account of final functors and their properties. In Section \ref{s3}, we show that a certain functor arising in the theory of locales is final.
This result will play a critical role in Section \ref{s4}, where we prove the existence and idempotence of the density comonad $T$ used in the definition of the category $\kHLoc$.  
In Section \ref{s5}, we define the category $\kHLoc$ of compactly generated strongly Hausdorff locales as the category of $T$-coalgebras and investigate some of its properties. In particular, we show that $\kHLoc$ is at least as large as the one defined by Escardó. We also use Hofmann-Lawson duality to prove that every continuous strongly Hausdorff locale is compactly generated. In the last two sections,  we prove that $\kHLoc$ is cartesian closed.
\subsection*{Conventions and notations}
Throughout this paper, the product of two categories $\ma$ and $\mb$ is denoted by $\ma \tm \mb$. {\bf A subcategory $\mb$ of a category $\mc$ is always assumed to be full.} Given two objects $X, Y \in \mc$, the set of morphisms from $X$ to $Y$ is denoted by $\mc(X, Y)$.

When it exists, the cartesian product of two objects $X$ and $Y$ in $\mc$ is denoted by $X \times_{\mc} Y$. If $X$ and $Y$ belong to a subcategory $\mb$ of $\mc$, their product $X \times_{\mb} Y$ in $\mb$ may differ from $X \times_{\mc} Y$ and should not be confused with it. Observe that if $X \times_{\mc} Y$ exists and is isomorphic to an object of $\mb$, then $X \times_{\mb} Y$ exists and $X \times_{\mathcal{B}} Y\cong X \times_{\mathcal{C}} Y$.

\section{Kan extendable subcategories}\label{s1}
 The purpose of this section is to recall the notion of Kan extendable subcategories and to collect their properties that will be used in the sequel. The presentation is arranged to be self-contained, so that reference to \cite{GM} is required only for the proofs. 
 
 While the codensity monad has been studied more extensively, our focus here is on its dual notion: the density comonad. We briefly recall its definition and principal properties. Detailed accounts of this notion can be found in Leinster \cite[Sections 2,5 and 6]{LT} and Dubuc \cite[page 68]{DE}.
 
 Let $F:\ma \to \mb$ be a functor and assume that $F$ has a left pointwise Kan extension along itself $(T,\eta)$. The density comonad of $F$ is the comonad $(T,\epsilon, \delta)$, where 
 \begin{itemize}
 	\item  Its counit is the natural transformation $\epsilon: T \Ra 1_{\mb}$ given by the universal property of $(T,\eta)$ with respect to the natural transformation
 	\begin{equation}\label{x165}
 		1_F: F \Ra 1_{\mb}F. 
 	\end{equation}
 	 That is, $\epsilon$    is the unique natural transformation from $T$ to $1_{\mb}$ satisfying
 	  \begin{equation}\label{x166}
 	  \epsilon F.\eta= 1_{F}. 
 	  \end{equation}
 	  
 	 \item Its comultiplication is the  natural transformation $\delta: T\Ra T^2$ given by the universal property of $(T,\eta)$ with respect to the composite natural transformation 
 	 \begin{equation}\label{x167}
 	    F \overset{\eta}{\Ra} TF \overset{T\eta}{\Ra} T^2F.  
 	 \end{equation}
 	That is, $\delta$    is the unique natural transformation from $T$ to $T^2$ satisfying
 	\begin{equation}\label{x168}
 	 \delta F.\eta= T\eta.\eta. 
 	\end{equation}
 	\end{itemize}
 	   Recall that a comonad $(T, \epsilon, \delta)$ of a category $\mc$ is said to be idempotent if its comultiplication $ \delta: T \ra T^2$ is an isomorphism.
 \begin{definition} \cite[Definition 3.1]{GM} \\
 	A subcategory $\mw$ of a category $\mc$ is said to be left Kan extendable if: 
 		\begin{enumerate}
 		\item  The inclusion functor $J:\mw \to \mc$ has a pointwise left Kan extension.
 		\item The density comonad  $(T,\epsilon, \delta)$ of the functor $J$ is idempotent.
 		\end{enumerate}
 	\end{definition}
 	Let $\mw$ be a left Kan extendable subcategory of $\mc$ and let $(T,\epsilon,\delta)$ be the  density comonad of the inclusion functor 
 	$J:\mw \ra \mc$.  The category  of $T$-coalgebras of $\mc$ is 
 	 denoted by $\mw_l[\mc]$. By \cite[Proposition 2.2]{GM}, the forgetful functor
 	 \begin{equation}\label{x169}
 	 U: \mw_l[\mc]\to \mc 
 	 \end{equation}
 	 is fully faithful and injective on objects. Therefore, the category $\mw_l[\mc]$ is viewed as a subcategory of $\mc$ which is replete. It is called the subcategory of $\mw$-generated objects of $\mc$. 
 	 \begin{proposition}\label{x85} \cite[Proposition 3.3]{GM} \\
 	 	Let $\mw$ be a left Kan extendable subcategory of  $\mc$, $(T,\epsilon,\delta)$ the density comonad of the inclusion functor $J:\mw \ra \mc$ and $U:\mw_l[\mc]\ra \mc$ the forgetful functor. Then: 
 	 	
 	 	\begin{enumerate}
 	 		
 	 		\item  The subcategory $\mw_l[\mc]$ is coreflective  in $\mc$.
 	 		
 	 		\item  The free $T$-coalgebra functor $F_{T}: \mc \ra \mw_l[\mc]$ is a coreflector.
 	 		\item  The coreflection $(U \dashv F_{T})$ has  $\epsilon$ as its counit. 
 	 		
 	 	\end{enumerate}
 	 \end{proposition}
 	 \begin{proposition}\label{x87} \cite[Proposition 3.4]{GM} \\
 		Let $\mw$ be a left Kan extendable subcategory of $\mc$. 
 		\begin{enumerate}
 			
 			\item The inclusion functor $\mw_l[\mc]\st{U}{\ra}\mc$ creates all colimits that $\mc$ admits.
 			\item The subcategory $\mw_l[\mc]$ has all limits that $\mc$ admits formed by applying the  coreflector $F_{T}$ to the limit in $\mc$.
 		\end{enumerate}
 		In particular, if $\mc$ is either complete, cocomplete or bicomplete, then so is $\mw_l[\mc]$.
 	\end{proposition}

 	\begin{corollary}\label{x96}\cite[Corollary 3.5]{GM} \\
 		 Let $\mw$ be a left Kan extendable subcategory of $\mc$ and $C\in\mc$. Then the following two properties are equivalent:
 		\begin{enumerate}
 			
 			\item The object $C$ is $\mw$-generated.
 			\item  There exists a functor $F:\mk\ra\mw$ such that $C\cong\colim JF$, where  $J:\mw\to\mc$ is the inclusion functor.
 		\end{enumerate}
 		\end{corollary} 
 	 Recall that an object $A$ of a category $\mc$ is said to be exponentiable if for each $X\in \mc$, the binary product $X\tm_{\mc}A$ exists and 
 the functor
 \begin{equation}\label{x170}
  -\tm_{\mc} A:\mc\ra\mc
 \end{equation}
   has a right adjoint.

 \begin{lemma}\label{x86}\cite[Lemma 9.2]{GM}\\
 	Let $\mw$ be a left Kan extendable subcategory of a bicomplete category $\mc$. Assume that  
 	\begin{enumerate}
 		\item  Every object in $\mw$ is exponentiable in $\mc$.
 		\item For all $V,W \in \mw$, the object $V\times_{\mc} W \in \mw_l[\mc]$.
 	\end{enumerate}
 	Then for every $V \in \mw$ and every $Y \in \mw_l[\mc]$, $V\times_{\mc} Y$ is a $\mw$-generated object. That is  $V\tm_{\mw_l[\mc]} Y\cong V\times_{\mc} Y$. 
 \end{lemma} 
  \noindent Assume next that $\mw$ and $\mc$ are as in Lemma \ref{x86}.  
 \begin{itemize}
 	\item For $X,Y \in \mw_l[\mc]$, let
 	\begin{equation}\label{x108}
 	 J_X: \mw/X \to \mc
 	\end{equation}
 	be the functor defined by $J_X(\sigma:V \sra X)=V$, and let 
 	$J_X \tm_{\mc} Y$ be the composite functor
 	\begin{equation}\label{x88}
 		 \mw/X \st{J_X}\ra \mc\xra{-\tm_{\mc} Y} \mc.
 	\end{equation}
 	By Proposition \ref{x87}, $\mw_l[\mc]$ is complete. For
 	$(V \st{\sigma} \ra X) \in \mw/X$, define

 		\begin{equation}\label{x116}
 		\theta_{\sigma}= \sigma \tm_{\mw_l[\mc]} 1_Y:V \tm_{\mc} Y = V \tm_{\mw_l[\mc]} Y\ra X \tm_{\mw_l[\mc]} Y.
 	\end{equation}

 	The maps $\theta_{\sigma}$ define a cone  
 	\begin{equation}\label{x89}  
 		J_X\tm_{\mc} Y\st{\theta}\Ra X \tm_{\mw_l[\mc]} Y. 
 	\end{equation}
 	\item For $V\in \mw$, let 
 	\begin{equation}\label{x109 }
 	\Hom(V,-): \mc\to \mc	
 	\end{equation}
 	 be a right adjoint of the functor
 	$-\tm_{\mc} V: \mc\ra \mc$. For $Y,Z \in \mw_l[\mc]$, define $S^{Y}_{Z}$
 	to be the functor
 	 \begin{equation}\label{x90}
 		\begin{array}{rccl}
 			S^{Y}_{Z} :& (\mw/Y)^{op}& \ra &\mc \\
 			&(V \st{\sigma} \ra Y)&\longmapsto & \Hom(V,Z).
 		\end{array}
 	\end{equation}
 \end{itemize}
 
 \begin{definition}\label{x91} \cite[Definition 9.3]{GM}\\
 	A left Kan extendable subcategory $\mw$ of a bicomplete category $\mc$ is said to be  \textbf{\textit{closeable}} if 
 	\begin{enumerate}

 		\item Every object in $\mw$ is exponentiable in $\mc$.
 		\item For all $V,W \in \mw$, the object $V\times_{\mc} W \in \mw_l[\mc]$.
 		\item For all $X,Y \in \mw_l[\mc]$, the cone  $J_X\tm_{\mc} Y\st{\theta}\Ra X \tm_{\mw_l[\mc]} Y$ given by (\ref{x89}) is a colimiting cone.
 		\item For all $Y,Z \in \mw_l[\mc]$, the functor  $S^{Y}_{Z}: (\mw/Y)^{op} \ra \mc$ given by (\ref{x90}) has a limit. 
 	\end{enumerate}
 \end{definition}
   \noindent Assume that $\mw$ is a closeable left Kan extendable subcategory of a bicomplete category $\mc$.
 Define 
 \begin{equation}\label{x92} 
 	\hom(-,-) : \mw_l[\mc]^{op}\times \mw_l[\mc] \ra  \mc
 \end{equation}
 by
 
 \begin{equation}\label{x117}
 \hom(Y,Z)= \lm S^{Y}_{Z}=\underset{(V\st{\sigma}{\rightarrow}Y)\in \mw/Y}\lm \Hom(V,Z).	
 \end{equation}
 Let $F_{T} :\mc \ra \mw_l[\mc]$ be the coreflector given by Proposition \ref{x85} and define 
 \begin{equation}\label{x93}
 	\begin{array}{rrll}
 		(-)^{(-)}:&\mw_l[\mc]^{op}\times \mw_l[\mc] &\ra &\mw_l[\mc]\\
 		&(Y,Z) &\mapsto &Z^Y
 	\end{array}
 \end{equation}
 to be the composite functor
 \begin{equation}\label{x94} 
 	\mw_l[\mc]^{op}\times \mw_l[\mc] \xra{\hom} \mc \st{F_{T}}\ra \mw_l[\mc],
 \end{equation} 
 so that  for all  $Y,Z \in \mw_l[\mc]$, $ Z^Y=F_{T}(\hom(Y,Z))$. 
 \begin{theorem}\label{x95} \cite[Theorem 9.6]{GM}\\
 	Assume that $\mw$ is a closeable left Kan extendable subcategory of a bicomplete category $\mc$. Then $\mw_l[\mc]$ is cartesian closed with internal hom functor  as defined by (\ref{x94}).
 \end{theorem}
 \section{Final functors}\label{s2}
This section introduces the notion of final functors and their properties. Further details and proofs of the stated results can be found in Mac Lane \cite[page 217]{ML},  Borceux \cite[Section 2.11]{BF}, Perrone-Tholen \cite[Section 3.2]{PT} and  Riehl \cite[Section 8.3]{RE}.

Let $X: \mj \to \mc$ be a functor and $C\in \mc$. Recall that a cone    $\lambda: X \Ra C$  is just a natural transformation from the functor $X$ to  the constant functor at $C$.\footnote{
In some references, a {\it cone} is defined as a natural transformation from a constant functor to an arbitrary functor, whereas a natural transformation from an arbitrary functor to a constant functor is called a {\it cocone}.}
It consists then of a class of maps $\lambda_j:X(j) \sra C$, $j\in \mj$   such that  for every map $\alpha: j\sra k$ in $\mj$, the following diagram commutes

 \begin{equation}\label{x110} 
 	 \begin{tikzcd}[ row sep=0.7 em  ]
 		X(j)\arrow[rd,"\lambda_j"]\arrow[dd,"X(\alpha)"']&\\
 		&C\\
 		X(k)\arrow[ru,"\lambda_k"']&
 	\end{tikzcd}
 \end{equation} 
The map $\lambda_j:X(j) \sra C$  is called the component of the cone $\lambda$ along the object $j \in \mj$. Such a cone can be either composed from the left by a map or from the right by a functor: 
 \begin{itemize}
\item If $h:C \sra C'$ is any map in $\mc$, then the composite of $h$ with $\lambda$ is the cone
$h\lambda:X\Ra C'$   defined by  $(h\lambda)_j=h\lambda_j$,  so that we have commutative diagrams

\begin{equation}\label{x111} 
	\begin{array}{cc}
		\begin{tikzcd}
			X(j) \arrow[dr, "\lambda_j"' ] \arrow[rr, "(h\lambda)_j" ]{}
			& & C'  \\
			& C\arrow[ur, "h"' ]    
		\end{tikzcd}
		&\begin{tikzcd}
			X \arrow[rr, Rightarrow, "h\lambda" ] \arrow[dr, Rightarrow, "\lambda"' ] 
			& & C'  \\
			& C\arrow[ur, "h"' ]    
		\end{tikzcd}
	\end{array}
\end{equation}

\item
Let $F\colon \mi \to \mj$ be any functor. The composite of $\lambda$ with $F$ (also called the restriction of $\lambda$ along $F$)  is the  cone $\lambda F: XF\Ra C$ defined by

\begin{equation}\label{x113} 
(\lambda F)_i=\lambda_{F(i)}: X(F(i)) \ra C,	
\end{equation} 
so that one has a commutative diagram
\begin{equation}\label{x114} 
\begin{tikzcd}
	\mi \arrow[rr, Rightarrow, "\lambda F" ] \arrow[dr,  "F"' ] 
	& & X  \\
	& \mj\arrow[ur, Rightarrow,"\lambda"' ]    
\end{tikzcd}	
\end{equation} 
\end{itemize}
 Recall that the  cone $\lambda: X\Ra L$ is said to be colimiting if for each  cone 
 $\mu: X\Ra C$, there exists a unique morphism $h:L \sra C$ such that $\mu= h\lambda$. 
 Assume now that  $F: \ma \to \mb$ and $X: \mb \to \mc$ are two functors such that both $X$ and $XF$ have colimits. Let  $\lambda: X\Ra \colim X$ and $\mu: XF \Ra \colim XF$ be  colimiting cones.  Then there exists a unique map 
\begin{equation}\label{x74}
h: \colim XF \ra \colim X
\end{equation}
rendering the following diagram commutative
\begin{equation}\label{x115} 
\begin{tikzcd}
	XF \arrow[rr, Rightarrow, "\lambda F" ] \arrow[dr, Rightarrow, "\mu"' ] 
	& &\colim X  \\
	& \colim XF \arrow[ur, "h"' ]    
\end{tikzcd}	
\end{equation} 
Recall that a category $\mathcal{C}$ is said to be \emph{connected} if it is nonempty, and 
for every pair of objects $X$ and $X'$ in $\mathcal{C}$, there exists a zigzag of morphisms
connecting $X$ and $X'$. That is, a finite sequence of objects and morphisms of the form
\begin{equation}\label{x118} 
\begin{tikzcd}[column sep=0.4cm, row sep=0.4cm]
	X_0=X \arrow[dr]	&   & X_2   \arrow[rd]\arrow[ld] & \cdots & X_{n-2} \arrow[rd]\arrow[ld]& &X_n=X' \arrow[ld]\\
	& X_1   &   & \cdots & & X_{n-1}&
\end{tikzcd}
\end{equation} 
The following result is well known.
\begin{proposition}\label{x69} 
		Let $F\colon \ma \to \mb$ be a functor. Then the following properties are equivalent. 
		\begin{enumerate}
	\item For every functor $X\colon \mb \to \mc$ admitting a colimit, the composite functor $XF$ has a colimit, and the canonical map 
			\begin{equation}\label{z119} 
				\colim XF \longrightarrow \colim X
			\end{equation} 
			given by (\ref{x74}) is an isomorphism.		 
\item For every functor $X\colon \mb \to \Set$ admitting a colimit, the composite functor $XF$ has a colimit, and the canonical map 
			\begin{equation}\label{y119} 
				\colim XF \longrightarrow \colim X
			\end{equation} 
			given by (\ref{x74}) is an isomorphism.

			\item For every functor $X\colon \mb \to \mc$, $X$ has a colimit iff $XF$ has a colimit; moreover, when these colimits exist, the canonical map
			\begin{equation}\label{x121} 
				\colim XF \longrightarrow \colim X
			\end{equation}
			is an isomorphism.
\item For every functor $X\colon \mb \to \Set$, $X$ has a colimit iff $XF$ has a colimit; moreover, when these colimits exist, the canonical map
			\begin{equation}\label{y121} 
				\colim XF \longrightarrow \colim X
			\end{equation}
			is an isomorphism.

			\item For every functor $X\colon \mb \to \mc$ admitting a colimit and every colimiting cone 
			\begin{equation}\label{x138} 
				\lambda:  X \Ra \colim X, 
			\end{equation} 
			the restriction
			\begin{equation}\label{x139} 
				\lambda F:  XF \Ra \colim X 
			\end{equation} 
			of $\lambda$ along $F$ is again colimiting.
\item For every functor $X\colon \mb \to \Set$ admitting a colimit and every colimiting cone 
			\begin{equation}\label{Y138} 
				\lambda:  X \Ra \colim X, 
			\end{equation} 
			the restriction
			\begin{equation}\label{x139} 
				\lambda F:  XF \Ra \colim X 
			\end{equation} 
			of $\lambda$ along $F$ is again colimiting.

			\item For every functor $X\colon \mb \to \mc$, a cone 
			\begin{equation}\label{x140} 
				\lambda:  X \Ra L, 
			\end{equation}
			is colimiting iff its restriction along $F$
			\begin{equation}\label{x141} 
				\lambda F:  XF \Ra L 
			\end{equation} 
			is  colimiting.

\item For every functor $X\colon \mb \to \Set$, a cone 
			\begin{equation}\label{y140} 
				\lambda:  X \Ra L  
			\end{equation}
			is colimiting iff its restriction along $F$
			\begin{equation}\label{x141} 
				\lambda F:  XF \Ra L 
			\end{equation} 
			is  colimiting.

			\item For each $B \in \mb$, the composite functor
			\begin{equation}\label{x120} 
				\ma \xrightarrow{F} \mb \xrightarrow{\mb(B,-)} \Set
			\end{equation} 
			has colimit $1$.
			
			\item For every $B \in \mb$, the slice category $B/F$ is connected.
			\end{enumerate}
	\end{proposition}
	\begin{definition}\label{x70} 
	A functor $F\colon \ma \to \mb$  is said to be final if it satisfies the equivalent properties in Proposition \ref{x69}.
		\end{definition}

\begin{proposition}\label{x5}
	Every right adjoint functor is final.
\end{proposition}
\begin{proof} A category having an initial object is connected.  
Let

\begin{equation}\label{x122} 
 F : \ma \rightleftarrows \mb:G  
\end{equation}
be an adjunction with unit $\eta: 1_{\ma}\Ra GF$ and let $A\in \ma$. The arrow-object $\eta_A: A\sra GF(A)$ of the slice category $A/G$ is initial. Therefore $A/G$ is connected. It follows that the functor $G$ is final. 
 \end{proof}
\begin{definition}\label{x71} 
A category $\ma$  is said to be finally small if there exists a final functor $F:\mi \to \ma$, where $\mi$ is small.
	\end{definition}
	 The following result is an obvious consequence of the last two definitions.
\begin{proposition}\label{x72} 
Every functor from a finally small category to a cocomplete category has a colimit.
		
\end{proposition}

 \begin{remark}\label{x73} 
Dually, a functor $F\colon \ma \to \mb$ is said to be cofinal if  the opposite functor $F^{op}\colon \ma^{op} \to \mb^{op}$ is final. The above properties for final functors have their duals for cofinal functors.
\end{remark}
\section{A useful example of final functors}\label{s3}
This section presents an example of a final functor arising in the theory of locales. This example will be used in the proof of one of the main results of the paper, where the established finality property will be essential. We begin with a brief summary of the necessary background on locales. For further reading and complete proofs of the underlying results, the reader is referred to the monographs by Picado and Pultr \cite{PP,PPS}, as well as to Chapter II of \cite{PPT} by Picado, Pultr, and Tozzi.

A \emph{locale}   is a complete lattice $L$ in which
\begin{equation}\label{x123} 
 a \wedge \bigvee B = \bigvee \{ a \wedge b \mid b \in B \}
	 \end{equation}
for all $a \in L$ and $B \subseteq L$.  Let $L$ and $M$ be two locales. A \emph{localic}  map from $L$ to $M$ is a meet-preserving function $f:L \ra M$ whose left adjoint preserves finite meets. The category of locales, denoted by $\Loc$, is the category whose objects are the locales and morphisms are the localic maps. By \cite[Chapter IV, Corollary 4.3.5]{PP} $\Loc$ is bicomplete.

A locale $L$ is a complete Heyting algebra with Heyting operator  given by
\begin{equation}\label{x124} 
a \Ra b = \bigvee \{ x \in L \mid a \wedge x \leq b \}.
\end{equation}
A subset \( S \) of  \( L \) is a \emph{sublocale} if it is a locale with the induced order and the inclusion map $S \sra L$ is localic. Equivalently, \( S \subseteq L \) is a sublocale if it satisfies the following two conditions:
\begin{equation}\label{x125} 
\bigwedge A \in S \quad \text{for all } A \subseteq S,
\end{equation}
\begin{equation}\label{x126} 
	x \Ra s \in S \quad \text{for all } x \in L,\ s \in S.
	\end{equation}

If $f: L\ra M$ is a localic map and $S$ is a sublocale of $L$ then the image of $S$ by $f$ which is 
\begin{equation}\label{x127} 
f(S)=\{f(s): s\in S\}
\end{equation}
is a sublocale of $M$. 
For $a\in L$, the sublocale
\begin{equation}\label{x128} 
  \mathfrak{c}(a)=\uparrow a=\{x\in L: a\leq x\}  
\end{equation}
is called the closed sublocale induced by $a$.

Let  $f: L\ra M$ be a localic map. Then 
\begin{itemize} 
\item $f$ is said to be closed if the image by $f$ of every  closed sublocale of $L$ is a closed sublocale of $M$. That is, 

\begin{equation}\label{x129} 
 f(\mathfrak{c}(a))=\mathfrak{c}(f(a)), \quad \forall a\in L. 
\end{equation}

\item The inverse image of a closed sublocale of $M$ is a closed sublocale of $L$. More precisely, let $f^*: M \ra L$ be the left adjoint of $f$. Then
 \begin{equation}\label{x100} 
 	f^{-1}(\mathfrak{c}(b))= \mathfrak{c}(f^*(b)), \quad \forall b\in M.
 \end{equation}
\end{itemize}
A locale $L$ is said to be strongly Hausdorff if the diagonal map 
\begin{equation}\label{x130} 
\bigtriangleup : L \ra L\tm_{Loc}L 
\end{equation}
is closed.
The subcategory of $\Loc$ of strongly Hausdorff locales is denoted by $\HLoc$.

A subset $A$ of $L$ is a cover of $L$ if $\bigvee A=1$.  The locale 
$L$ is compact if every cover of it has a finite subcover.
A closed sublocale of a compact locale is compact. 
Let $\KHLoc$ denote the subcategory of $\HLoc$ of compact strongly Hausdorff locales.

 A compact subspace of a Hausdorff topological space is closed. The following result shows that this fact has a counterpart in the point-free setting.
\begin{theorem}\label{x40}\cite[Chapter III, Theorem 9.2.1]{PPS}\\
	A compact sublocale of a strongly Hausdorff locale is closed.
\end{theorem}
Arbitrary intersections of sublocales of $L$ 
are again sublocales of $L$. Therefore the set $\msl(L)$ of sublocales of  $L$, ordered by inclusion, is a complete lattice.  
The meet of a family of sublocales $(S_j)_{j\in J}$ in the coframe $\msl(L)$
is its intersection:
\begin{equation}\label{x26}
\bigwedge_{j \in J} S_{j} =    \bigcap_{j \in J} S_{j}.   
\end{equation} 
Its join is given by: 
\begin{equation}\label{x27}
\bigvee_{j \in J} S_{j} =  \{ \bigwedge A \mid A \subseteq \bigcup_{j \in J} S_{j}  \}.
\end{equation}
The lattice $\msl(L)$ is actually a coframe \cite[Chapter III, Theorem 3.2.1]{PP}.

Assume for the remainder of this section that $L$ is a strongly Hausdorff locale.
Let $\mk(L)$ be the subset of $\msl(L)$  of compact sublocales of $L$. 
 Then $\mk(L)$ is ordered by inclusion and is viewed as a small thin category in the ordinary way.

   Define $\HLoc/L$ to be the over category associated to the object $L$ of $\HLoc$, and let
	$\KHLoc/L$ be the  subcategory of $\HLoc/L$ whose objects are arrows $\sigma:K \sra L$ with compact strongly Hausdorff domain $K$.   By \cite[Proposition 5.3.4]{PPS},
	  a sublocale of a strongly Hausdorff locale is strongly Hausdorff. Therefore there is a functor  
 
\begin{equation}\label{x6}
U_L:\mk(L) \to \KHLoc/L
 \end{equation}
which takes an object 
  $K$ of $\mk(L)$ to the inclusion map $i_K:K\sra L$, and a morphism $ K \ra K'$, which is an inclusion map, to the  commutative diagram:

  \begin{equation}\label{x131} 
  \begin{tikzcd}
  		K\arrow[dr, "i_K"' ] \arrow[rr, "\subseteq" ]{}
  		& & K'\arrow[dl, "i_{K'}" ]  \\
  		& L    
  	\end{tikzcd}
  \end{equation}
  which is a morphism in $\KHLoc/L$. The functor $U_L$ is fully faithful and injective on objects. The category $\mk(L)$ is then viewed as a   subcategory of $\KHLoc/L$.

\begin{lemma}\label{x3}
	The subcategory $\mk(L)$ of $\KHLoc/L$ is  reflective.
	\end{lemma}
	\begin{proof} 
	Let
	
	\begin{equation}\label{x132} 
	R_L: \KHLoc/L  \to \mk(L) 
	\end{equation}
	be the functor defined as follows:
	 For any object $(\sigma: K\sra L) \in  \KHLoc/L$, let $\sigma(K)$ be the image of $K$ by $\sigma$ and
	 define $R_L(\sigma)= \sigma(K)$. Let $\sigma: K\sra L$ and $\sigma': K'\sra L$ be objects in $\KHLoc/L$ and $f$ a morphism from $\sigma$ to $\sigma'$. Then $f$ is a localic map $f:K\ra K'$ rendering the following diagram commutative
	\begin{equation}\label{x133} 
	\begin{tikzcd}
		K\arrow[dr, "\sigma"' ] \arrow[rr, "f" ]{}
		& & K'\arrow[dl, "\sigma'" ]  \\
		& L    
	\end{tikzcd}
	\end{equation}
	We have $\sigma=\sigma'f$, therefore $\sigma(K)=\sigma'f(K)\subseteq \sigma'(K')$. Define  $R_L(f)$ to be  the inclusion map

\begin{equation}\label{x134} 
	R_L(f): \sigma(K) \ra \sigma'(K').
\end{equation}
	The functor $R_L$ is a reflection of $\KHLoc/L$  on  $\mk(L)$.
	\end{proof}
 \begin{proposition}\label{x12}
	The  functor $U_L:\mk(L) \to \KHLoc/L$ given by (\ref{x6}) is final. In particular, the category $\KHLoc/L$ is finally small.
	\end{proposition}
	\begin{proof}
	 By Proposition \ref{x5}  and Lemma \ref{x3}, the functor $U_L$ is final. The category $\mk(L)$ is small, therefore  $\KHLoc/L$ is finally small. 
\end{proof}


\section{The density comonad $T$}\label{s4}
Let 
 	\begin{equation}\label{x142}
 	J:\KHLoc\to \HLoc
 \end{equation}
 be the inclusion functor. The aim of this section is to show that the functor $J$ has an idempotent density comonad.

  \begin{proposition}\label{x1}
	The  functor $J$ has a density comonad.
	\end{proposition}
	\begin{proof}
	For $L\in \HLoc$, define
	$D_L:\KHLoc/L\to\KHLoc$ to be the functor which takes an arrow-object $\sigma:K\rightarrow L$ to its domain $K$. Let $J_L$ be the composite functor
	\begin{equation}\label{x8}
\KHLoc/L \st{D_L}\to\KHLoc \st{J}\to \HLoc.
\end{equation}
 By \cite[Chapter V, Corollary 6.4.3]{PP}, the category $\HLoc$ is cocomplete and by  Proposition \ref{x12}, $\KHLoc/L$ is finally small, thus by Proposition \ref{x72}, $J_L$
  has a colimit. It follows that the functor $J$ has a density comonad $(T,\epsilon,\delta)$.
\end{proof}
\begin{remark}\label{x18} 
Let $L$,  $J_L$ and $(T,\epsilon,\delta)$ be as in the previous proposition and $U_L:\mk(L) \to \KHLoc/L$ as in (\ref{x6}). Define $j_L$ to be the composite functor
\begin{equation}\label{x98}
j_L= J_L U_L:\mk(L) \ra \HLoc.
\end{equation}
Then by the proof of the preceding proposition, we have

\begin{equation}\label{x105}
	\colim j_L \cong T(L).  
\end{equation}
Let 
\begin{equation}\label{x99}
	 \alpha_L:j_L\Ra T(L)
\end{equation}
be a colimiting cone and let

\begin{equation}\label{x99}
\lambda_L:j_L\Ra L
\end{equation}
be the cone whose components are the inclusion maps.
Then clearly, by  the dual of (7) in \cite{GM} and by the finality of the functor $U_L$,  the counit $\epsilon_L:T(L) \sra L$ is the unique map rendering the following diagram commutative
\begin{equation}\label{x23}
\begin{tikzcd}
j_L \arrow[rr, Rightarrow, "\lambda_L"] \arrow[dr, Rightarrow,"\alpha_L"']
& &L \\
& T(L) \arrow[ur , rightarrow ,"\epsilon_L"']
\end{tikzcd}
\end{equation}

\end{remark}

\begin{lemma}\label{x97}
	Let $L$ be a strongly Hausdorff locale and $F:\mj \to \HLoc$ a functor. Assume that 
	\begin{enumerate}
		\item  $F(j)$ is a sublocale of $L$  for all $j\in \mj$.
		\item If $\alpha: j \sra k$ is a morphism in $\mj$, then $F(j)\subseteq F(k)$ and $F(\alpha):F(j)\sra F(k)$ is the inclusion map.
		\item The  cone $\lambda: F \Ra L$ whose components are the inclusion maps is  colimiting.
	\end{enumerate}
	Then in the coframe $\msl(L)$, we   have

	\begin{equation}\label{x135}
 L= \bigvee \{F(j), j\in \mj\}.   
	\end{equation}

\end{lemma}

\begin{proof}
	Define
	\begin{equation}\label{x136}
	\tilde{L}= \bigvee \{F(j), j\in \mj\}.  
	\end{equation}
	Then clearly, the colimiting cone $\lambda$ factors through the inclusion map $i:\tilde{L} \ra L $ as follows
	
	\begin{equation}\label{x137}
			\begin{tikzcd}
			F \arrow[rr, Rightarrow, "\lambda"] \arrow[dr, Rightarrow]
			& &L \\
			& \tilde{L} \arrow[ur , rightarrow ,"i"']
		\end{tikzcd}
	\end{equation}
	By \cite[Proposition 11.29]{AHS}, each colimit is an extremal epi-sink.
	The map $i$ is monic, therefore the inclusion map $i$ is an isomorphism. It follows that $\tilde{L}=L$.
\end{proof}

\begin{remark}\label{x10}

		Let $F:\mj \to \HLoc$ be a functor and $L$ be a strongly Hausdorff locale. Assume that $\lambda: F \Ra L$ is a colimiting  cone and  let $\mpp$ be the subset of $\msl(L)$ given by 
		$$\mpp=\{ \lambda_j(F(j)), j\in \mj\}.$$
		Then
		\begin{enumerate}
			\item  $\mpp$ is ordered by inclusion and may be viewed as a small thin category in the ordinary way.  
			
			\item Let $\phi: \mpp \sra \HLoc$ be the inclusion functor. Then the  cone $\phi \Ra L$ whose components are the inclusion maps is colimiting.
		\end{enumerate}
		 \end{remark}
Observe that by \cite[Proposition 5.3.4]{PPS}, a sublocale of a strongly Hausdorff locale is again  strongly Hausdorff.
\begin{lemma}\label{x13}
	Let $L$ be a strongly Hausdorff locale, $\mpp \subseteq \mr \subseteq \msl(L)$ and,  
	\[
\begin{array}{ccc}
 \phi :\mpp \to \HLoc & \text{and} & \psi :\mr \to \HLoc                    
\end{array}
\]
the inclusion functors.
	Assume that 
	
	\begin{enumerate}
	\item If $R\in \mr$, then $R$ is a closed sublocale of $L$.
	\item If $P\in \mpp$ and $R \in \mr$,    then  $P\cap R\in \mpp$.
	\item The  cone $\alpha:\phi \Ra L$ whose components are the inclusion maps is colimiting. 
	\end{enumerate}
	Then the   cone $\beta:\psi \Ra L$ whose components are the inclusion maps is  colimiting.

	\end{lemma}
	\begin{proof}
	Let $\mu: \psi \Ra M$ be a  cone. We need to prove that there exists a unique morphism $h:L \sra M$ such that the following diagram commutes
\begin{equation}\label{x15}
\begin{tikzcd}
\psi  \arrow[dr, Rightarrow,"\beta"'] \arrow[rr, Rightarrow, "\mu"]
& &M \\
& L \arrow[ur  ,"h"']
\end{tikzcd}
\end{equation}
Let $\lambda: \phi \Ra M$ be the restriction of $\mu$ to $\mpp$.
\begin{itemize}	
\item {\bf Uniqueness}: A map $h$ that makes the diagram (\ref{x15}) commutative also makes the following diagram commutative 
\begin{equation}\label{x16}
\begin{tikzcd}
\phi \arrow[rr, Rightarrow, "\lambda"] \arrow[dr, Rightarrow,"\alpha"']
& &M \\
& L \arrow[ur  ,"h"']
\end{tikzcd}
\end{equation}
The  cone $\alpha$ is colimiting, therefore $h$ is unique.
\item {\bf Existence:} let $h:L\sra M$ be the unique map making the diagram (\ref{x16}) commutative. We need to prove that with this map $h$, the diagram (\ref{x15}) commutes. Let $R\in \mr$ and $x\in R$. By Lemma \ref{x97},  $L=  \underset{P \in \mpp }{\bigvee} P$. Therefore there exists a family $(x_i)_{i\in I}$ of elements of $\underset{P\in \mpp}{\bigcup}P$ such that $x= \underset{i\in I}{\bigwedge}x_i$. For $i\in I$, let $P_i\in \mpp$
be such that $x_i\in P_i.$ Observe that:
\begin{itemize}
\item By condition 1 of the lemma, 
$x_i\in R$,   $\forall i\in I$.
   \item By condition 2 of the lemma, 
	 $P_i\cap R \in \mpp$, $\forall i\in I$. 
	\end{itemize}
	Therefore
$$ 
\begin{array}{rcl}
h\beta_R(x)&=&h\beta_R(\underset{i\in I}{\bigwedge}x_i)\\
 &=&\underset{i\in I}{\bigwedge} h\beta_R(x_i)\\
&=&\underset{i\in I}{\bigwedge} h\beta_{P_i\cap R}(x_i)  \\
&=&\underset{i\in I}{\bigwedge} h\alpha_{P_i\cap R}(x_i)  \\
 &=&\underset{i\in I}{\bigwedge} \lambda_{P_i\cap R}(x_i) \\
&=&\underset{i\in I}{\bigwedge} \mu_{P_i\cap R}(x_i) \\
 &=&\underset{i\in I}{\bigwedge} \mu_{R}(x_i) \\
  &=& \mu_{R}(\underset{i\in I}{\bigwedge}x_i) \\
&=& \mu_{R}(x). 
\end{array} 
$$
Thus the following diagram commutes
\begin{equation}\label{x152}
	\begin{tikzcd}
		\psi(R) \arrow[rr,  "\mu_R"] \arrow[dr, "\beta_R"']
		& &M \\
		& L \arrow[ur  ,"h"']
	\end{tikzcd}
\end{equation}
It follows that the diagram (\ref{x15}) commutes.
\end{itemize}	
\end{proof}
\begin{lemma}\label{x17}
	Let $L$ be a strongly Hausdorff locale, $\mq \subseteq \mk(L)$ and   
	 $\phi :\mq \to \HLoc$ 
		 the inclusion functor.  
	 Assume that
	  \begin{enumerate}
	  \item $Q\cap K\in \mq$ for all $Q\in \mq$ and all $K\in \mk(L).$ 
	  \item The  cone 
	    $\alpha:\phi \Ra L$, whose components are the inclusion maps, is  colimiting.
	  \end{enumerate}
	  Let $T$ be the density comonad given by Proposition \ref{x1}.
	  Then $T(L)=L$ and the counit $\epsilon_L:L \sra L$ is the identity map.
	\end{lemma}
	  \begin{proof}
	   Let 
	   \begin{equation}\label{x153}
	   j_L: \mk(L) \to \HLoc 
	   \end{equation}
	   be as in (\ref{x98}). 
	    By Theorem \ref{x40}, a compact sublocale of a strongly Hausdorff locale is closed. Moreover, the cone  $\alpha:\phi \Ra L$ is colimiting.
		Therefore by Lemma \ref{x13}, the  cone   
		\begin{equation}\label{x35}
		 j_L \Ra L 
		 \end{equation}
		  whose components are the inclusion maps is colimiting. 
		  By Remark \ref{x18}, $T(L)=L$ and the counit $\epsilon_L:L \sra L$ is the identity map.
		  \end{proof}

	\begin{proposition}\label{x7}
	The density comonad $T$ of the inclusion functor
	
	\begin{equation}\label{x155}
		J:\KHLoc\to \HLoc
	\end{equation}

	  is idempotent.
	\end{proposition}
	\begin{proof}
	Let $L$ be a strongly Hausdorff locale, $j_L:\mk(L) \to \HLoc $ be as in  (\ref{x98}) and
	\begin{equation}\label{x38}
	\alpha_L: j_L\Ra T(L)
	\end{equation}
	    a colimiting cone.
	Define
	\begin{equation}\label{x60}
		 \mq=\{(\alpha_L)_{K}(K), K\in \mk(L)\}.
\end{equation}
	Then $\mq \subseteq \mk(T(L))$ and by Remark \ref{x10}, the cone $\mq \Ra T(L)$ whose components are the inclusion maps is colimiting. Moreover, by (\ref{x100}),
	$Q\cap K\in \mq$ for all $Q\in \mq$ and all $K\in \mk(T(L)).$
 Therefore by Lemma \ref{x17}, $T^2(L)=T(L)$  and 
 \begin{equation}\label{x154}
 	\epsilon_{T(L)}:T^2(L)=T(L) \sra T(L) 
 \end{equation}
 is the identity map. It follows that the comonad $T$ is idempotent.
\end{proof}

\section{The category $\kHLoc$}\label{s5}

This section is devoted to defining compactly generated strongly Hausdorff locales and studying their properties.

 By Propositions \ref{x1} and \ref{x7}, the subcategory $\KHLoc$ of $\HLoc$ is left Kan extendable. Define the category of compactly generated strongly Hausdorff locales $\kHLoc$ by
 
 \begin{equation}\label{x156}
 \kHLoc=(\KHLoc)_{l}[\HLoc]. 
 \end{equation}

By definition, $\kHLoc$ is the category of coalgebras over the density comonad $T$.  We next compare  this category with  the category of compactly generated strongly Hausdorff locales defined by Escardó  in   \cite{EM}.

\begin{remark}\label{n3} 
Let $L \in \HLoc$, $j_L:\mk(L) \ra \HLoc$, 
$\alpha_L:j_L\Ra T(L)$,
		$\lambda_L:j_L\Ra L$ and $\epsilon_L:T(L) \sra L$ be as in Remark \ref{x18}. Let $E: \HLoc \to \Loc$ be the inclusion functor and define 
 \begin{itemize}
 	\item $\tilde{j_L}=Ej_L:\mk(L) \to \Loc$.	
 	\item $\tilde{\alpha}_L:\tilde{j_L} \Ra \colim \tilde{j_L}$ a colimiting cone.
 	\item $\tilde{\lambda}_L:\tilde{j_L} \Ra L$	the cone whose components are the inclusion maps.
 	\item $\tilde{\epsilon}_L: \colim \tilde{j_L} \ra L$ the unique map satisfying $\tilde{\epsilon}_L \tilde{\alpha}_L=  \tilde{\lambda}_L$.
 \end{itemize}	
Then 		
\begin{enumerate}
			\item By the dual of \cite[Proposition 2.3]{GM}, $L\in \kHLoc$   iff $\epsilon_L$ is an isomorphism.	
			\item $L$ is compactly generated in the sense of Escardó  
			 iff the map $\tilde{\epsilon}_L$ is an isomorphism (\cite[page 148]{EM}).
			\item The subcategory $\HLoc$ of $\Loc$ is reflective. Therefore, any strongly Hausdorff locale that is compactly generated in the sense of Escardó is in $\kHLoc$. In other words,  the category of compactly generated strongly Hausdorff locales defined by Escardó is a subcategory of $\kHLoc$.
			
		\item It seemed reasonable, though not clear, to Escardó that $\colim \tilde{j_L}$ is strongly Hausdorff for every strongly Hausdorff locale $L$ (\cite[Question 1.2, p.~151]{EM}). This is not clear to us either. If this claim holds, then $\colim   j_L \cong \colim \tilde{j_L}$.  
		In particular, $\epsilon_L$ is an isomorphism if and only if $\tilde{\epsilon}_L$ is. In this case, the category of compactly generated strongly Hausdorff locales  of Escardó is precisely $\kHLoc$.
	\end{enumerate}	
\end{remark}

The reader may compare the following result with (\cite[Proposition 5.7]{EM}).
\begin{proposition}\label{x2}
	The subcategory $\kHLoc$ of $\HLoc$ is coreflective.
	\end{proposition}
	\begin{proof}
	This is a consequence of Proposition \ref{x85}(1).
	\end{proof}
	
Let $k: \HLoc \ra \kHLoc$ be  the coreflection.
	Then 
	\begin{equation}\label{x59}
		k(L) =T(L)\cong \colim J_L,  \; \; \;\forall L\in \HLoc. 
	\end{equation}

	\begin{proposition}\label{x25}
	A strongly Hausdorff locale $L$ is compactly generated iff there exists a functor $F:\mj\to\KHLoc$ such that $L\cong\colim JF$, where  $J:\KHLoc \to \HLoc$ is the inclusion functor.
	\end{proposition}
	\begin{proof}
		This is a consequence of Corollary \ref{x96}.
	\end{proof}
Let $L$ be any locale.	
\begin{itemize}
	\item  The negation  (or pseudocomplement) of an element $a$ in  $L$ is defined as 
	\begin{equation}\label{x157}
		a^*  = \bigvee \{ x \in L \mid a \wedge x =0 \}.
	\end{equation}
	\item  The “\emph{rather-below}” relation, denoted by “$\prec$”, is defined on $L$ by
	
	\begin{equation}\label{x158}
			a\prec b  \Leftrightarrow a^*\vee b=1. 
	\end{equation}
	\item The locale $L$ is said to be regular if
	\begin{equation}\label{x159}
	  a=\bigvee \{x\in L, x\prec a\}, \quad \forall a\in L. 
	\end{equation}

\end{itemize}

Let $\Top$ be the category of topological spaces  and let
\begin{equation}\label{x78}
		Lc : \Top \rightleftarrows \Loc : Sp 
\end{equation}
	be the standard  adjunction (\cite[Theorem on page 58]{PPT}).
\begin{remarks}\label{x84} Recall that
	\begin{enumerate}

	\item A space $X$ is regular iff the locale $Lc(X)$ is a regular locale.
	\item A space $K$ is compact  iff the locale $Lc(K)$ is a compact locale.

	\item A  regular locale is strongly Hausdorff (\cite[Chapter V, Proposition  5.4.2]{PP}).
	\item A strongly Hausdorff locale which is either compact or continuous is regular (\cite[Proposition 4.8, page 45]{MV}).
	\item The category $\HLoc$ is a reflective subcategory of  $\Loc$  \cite[Chapter V, Proposition  6.4.2]{PP}.

	\end{enumerate}
\end{remarks}
\begin{proposition}\label{x83}
	Let $K$ be a compact Hausdorff  space. Then $Lc(K)$ is a compact strongly Hausdorff locale.
\end{proposition}
\begin{proof}
	The space $K$ is compact Hausdorff.   By \cite[Theorem 32.3]{MJ}, $K$ is normal Hausdorff and therefore a regular  space. By Remarks \ref{x84}.1 and \ref{x84}.2, $Lc(K)$   is  a compact regular locale and by Remark \ref{x84}.3, $Lc(K)$ is a compact strongly Hausdorff locale.	
\end{proof}
Our next result relies on the Hofmann-Lawson duality (\cite[Chapter VII, Theorem 6.4.3]{PP}), which we now introduce.	
	\begin{theorem}\label{x101} 
		The adjunction $(Lc,Sp)$ given by (\ref{x78}) induces an equivalence between the category of continuous locales and the category of sober locally compact spaces.
	\end{theorem}
	
	
	Let $L$ be a locale. Recall that
	\begin{itemize}
		\item The “\emph{well-below}” relation “\emph{$ \ll $ }”  is defined on $L$ by 
		\begin{equation}\label{x160}
			x \ll y \quad \text{iff} \quad 
			\Big(\forall D \subseteq L \text{ directed},\;
			y \leq \bigvee D \implies \exists d \in D,\; x \leq d\Big).
		\end{equation}
		\item The locale $L$ is said to be \emph{continuous} if
		\begin{equation}\label{x161}
				\forall a\in L, \quad a= \bigvee\{x\in L: x\ll a\}.
		\end{equation}

	\end{itemize}
	
	\begin{theorem}\label{x77}
	Every  strongly Hausdorff continuous locale $L$ is compactly generated.
	\end{theorem}
	\begin{proof}
	
	 Let $L$ be a  strongly Hausdorff continuous locale and 
	let $X$ be a sober locally compact space corresponding to $L$ under the 
	Hofmann-Lawson Duality.
	We have 
	\begin{equation}\label{x79}
		Lc(X)\cong  L.
	\end{equation}
By Remark \ref{x84}.4, $L$ is a regular locale. Therefore by Remark \ref{x84}.1, $X$ is a regular   space. Moreover, $X$ is a sober   space and therefore $X$ is $T_0$. A $T_0$ regular space is Hausdorff, thus $X$ is Hausdorff. Define $\mk(X)$ to be the category whose objects are the compact subspaces of $X$ and whose morphisms are the inclusion maps. Let 
\begin{equation}\label{x162}
	\phi:  \mk(X) \to \Top
\end{equation}
 be the functor which takes a compact subspace $K$ of $X$ to $K$ itself. Then clearly,
	
	\begin{equation}\label{x81}
	\colim \phi \cong X.
	\end{equation}
	By Remark \ref{x84}.5, $\HLoc$ is a reflective subcategory of  $\Loc$. 
	Let 
	\begin{equation}\label{x80}
		\psi: \Loc \to \HLoc 
	\end{equation} 
	be a reflector.
	We have
	\begin{equation}\label{x82}
	 \begin{array}{lllr}
	 	L&\cong&\psi (L)& \text{ because } $L$ \text{ is strongly Hausdorff}\\
	 	&\cong&\psi Lc(X)& \text{by (\ref{x79})}\\
	    &\cong&\psi Lc(\colim \phi)&\text{by (\ref{x81})}\\
	 	&\cong& \colim\psi Lc \phi& \text{ because } \psi \text{ and } Lc \text{ preserve colimits}. 
	 \end{array}  
	\end{equation} 
	Let $K\in \mk(X)$. The space $K$ is compact Hausdorff. By Proposition \ref{x83}, $Lc(K)$ is a compact strongly Hausdorff locale. Therefore,
 $\psi Lc(K) \cong Lc(K) \in \KHLoc.$ Thus  
$\psi Lc \phi(K)= \psi Lc(K) \cong Lc(K) \in \KHLoc$. It follows that the functor $\psi Lc \phi$ factors through the inclusion functor $J:\KHLoc\to \HLoc$. By (\ref{x82}), $L\cong \colim \psi Lc \phi$, therefore by Proposition \ref{x25}, $L\in \kHLoc$.
\end{proof}
\section{The cartesian closed structure on $\kHLoc$}\label{s6} 
In this section,  we use Theorem \ref{x95} to prove that the category $\kHLoc$ is cartesian closed.
	
We begin by recalling the following theorem, due to Hyland \cite[Theorem~1]{HM}.
		
\begin{theorem}\label{x31} The exponentiable objects in the category  $\Loc$ are precisely the continuous   locales.
\end{theorem}
The following result is due to Johnstone \cite[Proposition 3.2 page 104]{J}.	
\begin{proposition}\label{x48} If $L$ is a continuous locale and $M$ is a strongly Hausdorff locale, then the exponential object $\Hom(L,M)$ is strongly  Hausdorff. 
\end{proposition}
\begin{remark}
\begin{enumerate}
\item Johnstone's proposition asserts that 
 $\Hom(L,M)$ is strongly Hausdorff provided that 
$L$ is continuous and open, and 
$M$ is strongly Hausdorff. Classically, every locale is open \cite[page 97]{J}. This is not known to be constructively true. Consequently, Corollary \ref{x30} and Theorem \ref{x46} below inherit this non-constructive dependency. Whether Proposition \ref{x48} admits a constructive proof remains an open problem.

\item It is worth noting that by Remark \ref{n3}, Escardó's category is a subcategory of $\kHLoc$  and is cartesian closed, but it is not known whether it is coreflective in $\HLoc$. Conversely, $\kHLoc$ is coreflective in $\HLoc$ (Proposition \ref{x2}), but its cartesian closedness (Theorem \ref{x46}) is currently established only non-constructively. If we could prove that the two categories agree, we would have a category of compactly generated strongly Hausdorff locales that is at once coreflective in 
$\HLoc$ and constructively cartesian closed.
 \end{enumerate}

\end{remark}

\begin{corollary}\label{x30}A strongly Hausdorff compact locale is exponentiable in the category $\HLoc$.
	\end{corollary}
	\begin{proof}
	By Remark \ref{x84}.4, any strongly Hausdorff compact locale is regular.
	According to \cite[Chapter VII, Proposition 5.2.2]{PP}, every regular compact locale is continuous.
	Hence, the desired result follows from Proposition \ref{x48}.
	\end{proof}	 
Let $L$ and $M$ be two compactly generated strongly Hausdorff locales. Define
\begin{equation}\label{x33}
 S^{L}_{M}:(\KHLoc/L)^{op} \to \HLoc
\end{equation}
to be the functor given by 	
$S^{L}_{M}(\sigma:K \sra L)=\Hom(K,M)$, where $\Hom(K,M)$ is  the exponential object given by Proposition \ref{x48}.

\begin{lemma}\label{x32} 
The functor $S^{L}_{M}$ given by (\ref{x33}) has a limit.
	\end{lemma}
	\begin{proof}
This is a consequence of Propositions \ref{x72}, \ref{x12}  and the duality between limits and colimits.	 
	\end{proof}	
	 
\begin{proposition}\label{x102} 
The subcategory $\KHLoc$ is reflective in $\HLoc$.	
\end{proposition}
\begin{proof}
	Let $\CRLoc$ be the category of completely regular locales defined as in \cite[Chapter V. 6.3]{PP}.
	By \cite[page 134]{PP}, the category  $\KRLoc$ of compact regular locales is a reflective subcategory of $\CRLoc$ and by \cite[Proposition in page 94]{PP}, $\CRLoc$ is reflective in $\Loc$. Thus,  $\KRLoc$ is reflective in $\Loc$,  and is contained in $\HLoc$ by Remark \ref{x84}.3.  It follows that
	$\KRLoc$ is reflective in $\HLoc$.
	By Remark \ref{x84}.4,  $\KRLoc=\KHLoc$. Therefore $\KHLoc$ is reflective in $\HLoc$.
	 \end{proof}	
 \begin{corollary}\label{x49} 
 Let $K$ and $K'$ be two strongly Hausdorff compact locales.
 Then the binary product $K\tm_{\HLoc} K'$ is a compact locale.
 \end{corollary}
 \begin{proof}
By Proposition \ref{x102}, the inclusion functor  $\KHLoc \to \HLoc$ preserves products.	Therefore 
\begin{equation}\label{x163}
K\tm_{\HLoc}\K'\cong K\tm_{\KHLoc}\K'\in \KHLoc .
\end{equation}
	
 \end{proof}	
 Let $L$ and $L'$ be two compactly generated strongly Hausdorff locales and let
 \begin{equation}\label{x164}
 J_L: \KHLoc/L \ra \HLoc 
 \end{equation}
 be the functor given by (\ref{x8}).
 Define $J_L\times_{\HLoc} L'$ to be the composite functor 
 \begin{equation}\label{x41}
 	\KHLoc/L \xra{J_L}\ \HLoc \xrightarrow{-\tm_{\HLoc}L'} \HLoc. 
 \end{equation}
 We then have  
  \begin{equation} \label{x42}
 \begin{array}{rcll}
 	J_L\times_{\HLoc} L'(\sigma: K \sra L)&= &K\tm_{\HLoc}L'\\
 	&\cong&  K\tm_{\kHLoc}L' & \text{ by Lemma } \ref{x86}.
 \end{array}
\end{equation}
 Let 
 \begin{equation}\label{x43}  
 	\theta: J_L\tm_{\HLoc} L'\Ra L \tm_{\kHLoc} L'
 \end{equation}
 be the cone whose component along $ \sigma:K \sra L $ is the map 
 \begin{equation}\label{x44}  
 	\theta_{\sigma}: K\tm_{\HLoc}L' \cong K\tm_{\kHLoc}L' \xra{\sigma \tm_{kHLoc}1_{L'}}  L\tm_{\kHLoc}L'.
 \end{equation}
 \begin{lemma}\label{x45} 
 	The cone $\theta$ given by (\ref{x43}) is colimiting.
 \end{lemma}
The proof of this lemma is deferred to the next section.

\begin{theorem}\label{x46} 
The category $\kHLoc$ is cartesian closed.
	\end{theorem}
	\begin{proof}
		By  Corollaries \ref{x30}, \ref{x49} and Lemmas \ref{x32},	\ref{x45},  the subcategory $\KHLoc$ of  $\HLoc$ is closeable. 
		By Theorem \ref{x95}, the category $\kHLoc$ is cartesian closed. 
	\end{proof}

\section{The proof of Lemma \ref{x45}}\label{s7}
The proof of Lemma \ref{x45} relies on Fubini's theorem which is presented in the literature under its (co)end form (\cite[Proposition on page 230]{ML} and \cite[Theorem 1.3.1]{LF}). 

\bigskip
\noindent Let 
$B: \mi \tm \mj \to \mc$ be a functor such that for each $i\in \mi$, the partial functor

\begin{equation} 
 B(i,-):\mj \to \mc	 
\end{equation}
has a colimit. Let 
 \begin{equation}\label{nn1} 
 	\hat{B}: \mi \to \mc  
 \end{equation}
  be the functor defined by	$\hat{B}(i)= \colim B(i,-)$ and let
\begin{equation} 
		\epsilon^i:B(i,-) \Longrightarrow \hat{B}(i)
\end{equation}
be a colimiting cone. Let

 \begin{equation} 
 	\lambda: B \Longrightarrow l 
 \end{equation}
 be  any cone. For $i\in \mi$,  let
   \begin{equation} 
    \lambda_{(i,-)}:B(i,-)\Longrightarrow l
   \end{equation}
  be the restriction of the cone $\lambda$ along the functor
\begin{equation} 
    \mj \longrightarrow \mi\tm \mj 
   \end{equation}
which takes an object $j$ to $(i,j)$.
Let
    \begin{equation} 
   	\hat{\lambda}_{i}: \hat{B}(i) \ra l
   \end{equation}
    be the unique map rendering the following  diagram commutative

	\begin{equation} 
\begin{tikzcd}
B(i,-)  \ar[rr, Rightarrow,"\lambda_{(i,-)}" ] 
\ar[dr, Rightarrow, "\epsilon^i"']&& l\\
&\hat{B}(i)  \ar[ur, "\hat{\lambda}_{i}"']   
	\end{tikzcd}
\end{equation}
The maps $\hat{\lambda}_{i}$   are the components of a cone
\begin{equation} 
	\hat{\lambda}:\hat{B} \Longrightarrow l.
\end{equation}
We then have a correspondence
\begin{equation}\label{n1}   
	 \lambda \longrightsquigarrow \hat{\lambda}  
\end{equation}
from the class of cones  from $B$ to $l$ to the   class of cones from $\hat{B}$ to $l$.
Clearly, this correspondence is bijective with inverse the correspondence which takes a cone 
\begin{equation} 
 	\mu: \hat{B} \Longrightarrow l 
 \end{equation}
 to the cone 
 \begin{equation} 
 	\tilde{\mu}: B \Longrightarrow l,  
 \end{equation}
 whose component along $(i,j)$ is the composite 
 \begin{equation} 
 	  B(i,j) \xra{\epsilon^i_j} \hat{B}(i) \xra{\mu_i} l. 
 \end{equation}
  These inverse bijective correspondences are natural in $l$, and can therefore be easily used to prove the next theorem.
 \begin{theorem}\label{x143}(Fubini's theorem)\\
 Let $B: \mi \tm \mj \to \mc$ be a functor such that for each $i\in \mi$, the partial functor
\begin{equation} 
 B(i,-):\mj \to \mc	 
\end{equation}
has a colimit.
Let
\begin{equation} 
 	\hat{B}: \mi \to \mc  
 \end{equation}
be defined as in (\ref{nn1}).
Then a cone 
\begin{equation} 
 	\lambda: B \Longrightarrow l 
 \end{equation}
is colimiting iff the corresponding cone  

\begin{equation} 
	\hat{\lambda}:\hat{B} \Longrightarrow l
\end{equation} under the bijection (\ref{n1}) is colimiting.
In particular, 
\begin{enumerate}
\item The functor $B$ has a colimit iff the functor $\hat{B}$ has a colimit.

\item If $B$  has a  colimit, then
 
 \begin{equation} 
    \underset{(i,j)  \in \mi \tm \mj}{\colim}B(i,j)\cong  \underset{i\in \mi }{\colim} \; \underset{j\in \mj } {\colim}B(i,j).
 \end{equation}

 \end{enumerate}
\end{theorem} 


Let $L$ and $L'$ be two compactly generated strongly Hausdorff locales. Define a functor 
	 \begin{equation}\label{x52}
	 G:\mk(L) \tm \mk(L') \ra  \mk(L\tm_{\HLoc} L')
 	 \end{equation}
 	 by $G((K,K'))= K\tm_{\HLoc} K'$.

	\begin{lemma}\label{x51} 
	The functor $G$	given by (\ref{x52}) is final. 
	\end{lemma}
	
	\begin{proof} 
	Let p: $L\tm_{\HLoc} L' \ra L$ and p': $L\tm_{\HLoc} L' \ra L'$ be the projections. 
	Let 
	\begin{equation}\label{x53}
		F:   \mk(L\tm_{\HLoc} L') \ra \mk(L) \tm \mk(L')  
	\end{equation}
	be the functor given by
	\begin{equation}\label{x104}
		F(Q)=(p(Q),p'(Q)), \quad Q \in \mk( L\tm_{\HLoc} L').
	\end{equation}
The functor $F$ is a left adjoint of $G$. By Proposition \ref{x5}, $G$ is final.
\end{proof}
\noindent\textbf{Proof  of Lemma \ref{x45}:}\\

Let $k: \HLoc \ra \kHLoc$ be the coreflection given by (\ref{x59}).
Let  $L,L' \in \kHLoc$.
$$
\begin{array}{rclr}
	L\tm_{\kHLoc}L' & \cong &k(L\tm_{\HLoc}L')& \text{because }  k  \text{ preserves products}\\
	& \cong &T(L\tm_{\HLoc}L')& \text{ where } T \text{ is the density comonad of }  J \\
	&\cong&\colim j_{L\tm_{\HLoc}L'}& \text{by }(\ref{x105}) \\
	&\cong&\colim j_{L\tm_{\HLoc}L'}G& \text{because by Lemma   \ref{x51}, }  G  \text{ is final}\\
	&\cong&\colim \Psi & \text{where } \Psi \text{ is the composite } \Psi = j_{L\tm_{\HLoc}L'}G.
\end{array}
$$
Fix $K\in \mk(L)$.  The partial functor 

\begin{equation}\label{x62}
	\Psi(K,-): \mk(L') \to \HLoc 
\end{equation}
is the composite
\begin{equation}\label{x63}
	\mk (L') \st{j_{L'}}\to\HLoc \xra{K\tm_{\HLoc }-} \HLoc.
\end{equation}
The locale $L'$ is compactly generated strongly Hausdorff, therefore $\colim j_{L'}\cong L'$. The object  
$K$  is exponentiable in $\HLoc$, thus the functor
\begin{equation}\label{x64}
K\tm_{HLoc}- : \HLoc \to \HLoc
\end{equation}
commutes with colimits. It follows that $\Psi(K,-)$ has a colimit and
\begin{equation}\label{x65}
\colim \Psi(K,-) \cong K\tm_{\HLoc }L'.
\end{equation}
Let
\begin{equation}\label{x66}
	\hat{\Psi}: \mk(L) \to \HLoc 
\end{equation}
be the functor given by
\begin{equation}\label{x67}
	\hat{\Psi}(K)= \colim \Psi(K,-) \cong K\tm_{\HLoc }L'.
\end{equation}
By Fubini's theorem \ref{x143},  $\hat{\Psi}$ has a colimit and 
\begin{equation}\label{x68}
\colim \hat{\Psi} \cong \colim \Psi \cong  L\tm_{\kHLoc}L'.
\end{equation}
Moreover, the cone 
\begin{equation}\label{x106}
 \hat{\Psi} \Ra  L\tm_{\kHLoc}L' 
\end{equation}
whose component along an object $K \in \mk(L)$ is the inclusion map
\begin{equation}\label{x107}
	  K\tm_{\HLoc }L' =  K\tm_{\kHLoc }L'\ra L\tm_{\kHLoc}L'
\end{equation}
is colimiting. The cone (\ref{x106}) is the restriction of the cone (\ref{x43}) along the final functor $U_L$ given by (\ref{x6}). Therefore the cone  (\ref{x43}) is colimiting.\\

\noindent {\bf Funding}: This research received no external funding.



\begin{thebibliography}{23}

\bibitem{AHS}
Ad\'amek, J., Herrlich, H., Strecker, G.: Abstract and Concrete Categories.
  Wiley, New York (1990). Reprinted as: Repr. Theory Appl. Categ. \textbf{17},
  1--507 (2006)

\bibitem{BH}
Baez, J.C., Hoffnung, A.E.: Convenient categories of smooth spaces. Trans.
  Amer. Math. Soc. \textbf{363}(11), 5789--5825 (2011)

\bibitem{BF}
Borceux, F.: Handbook of Categorical Algebra 1, Basic Category Theory.
  Encyclopedia of Mathematics and its Applications, vol.~50. Cambridge
  University Press, Cambridge (1994)

\bibitem{BR}
Brown, R.: Some problems of algebraic topology. PhD thesis, Oxford
  University (1961)

\bibitem{DE}
Dubuc, E.J.: Kan Extensions in Enriched Category Theory. Lecture Notes in Mathematics, vol. 145. Springer, Berlin-Heidelberg-New York (1970)


  

\bibitem{EM}
Escard\'o, M.: Compactly generated Hausdorff locales. Ann. Pure Appl. Log.
  \textbf{137}(1--3), 147--163 (2006)

\bibitem{GM}
Ghazel, M.: Kan extendable subcategories and fibrewise topology. Adv. Pure
  Appl. Math. \textbf{15}(4), 21--83 (2024)

\bibitem{HSS}
Hovey, M., Shipley, B., Smith, J.: Symmetric spectra. J. Amer. Math. Soc.
  \textbf{13}(1), 149--208 (2000)

\bibitem{HM}
Hyland, M.: Function spaces in the category of locales. In: Continuous
  Lattices, Lecture Notes in Mathematics, vol.~871, pp. 264--281. Springer,
  Berlin (1981)

\bibitem{J}
Johnstone, P.: Open locales and exponentiation. In: Mathematical Applications of Category Theory, Contemp. Math., vol. 30, pp. 84-116. Amer. Math. Soc., Providence, RI (1984)

\bibitem{LT}
Leinster, T.: Codensity and the Ultrafilter Monad. Theory Appl. Categ.
  \textbf{28}(13), 332--370 (2013)

\bibitem{LF}
Loregian, F.: (Co)end Calculus, vol.~468. Cambridge University Press,
  Cambridge (2021)

\bibitem{ML}
Mac Lane, S.: Categories for the Working Mathematician, 2nd edn. Graduate
  Texts in Mathematics, vol.~5. Springer, New York (1998)

\bibitem{MV}
Moerdijk, I., Vermeulen, J.J.C.: Proper Maps of Toposes. Memoirs of the
  American Mathematical Society, vol.~148, no.~705 (2000)

\bibitem{MJ}
Munkres, J.: Topology, 2nd edn. Prentice Hall, Upper Saddle River, NJ (2000)

\bibitem{PT}
Perrone, P., Tholen, W.: Kan Extensions are Partial Colimits. Appl. Categ.
  Struct. \textbf{30}, 685--753 (2022)

\bibitem{PP}
Picado, J., Pultr, A.: Frames and Locales: Topology without Points.
  Frontiers in Mathematics. Birkh\"auser, Basel (2012)

\bibitem{PPS}
Picado, J., Pultr, A.: Separation in Point-free Topology. Birkh\"auser, Cham
  (2021)

\bibitem{PPT}
Picado, J., Pultr, A., Tozzi, A.: Locales. In: Categorical Foundations:
  Special Topics in Order, Topology, Algebra and Sheaf Theory, Encyclopedia
  of Mathematics and Its Applications, vol.~97, pp. 49--101. Cambridge
  University Press, Cambridge (2004)

\bibitem{RC}
Rezk, C.: Compactly generated spaces. Preprint available at
  \url{https://rezk.web.illinois.edu/cg-spaces-better.pdf}

\bibitem{RE}
Riehl, E.: Categorical Homotopy Theory. New Mathematical Monographs,
  vol.~24. Cambridge University Press, Cambridge (2014)

\bibitem{SN}
Steenrod, N.: A convenient category of topological spaces. Michigan Math.
  J. \textbf{14}, 133--152 (1967)

\bibitem{SP}
Strickland, N.P.: The category of CGWH spaces. Preprint available at
  \url{https://ncatlab.org/nlab/files/StricklandCGHWSpaces.pdf} (2009)

\end{thebibliography}
\end{document}